\documentclass{birkjour}
\usepackage{amsmath, amsthm, amscd, amsfonts, amssymb, graphicx, color}
\usepackage[bookmarksnumbered, colorlinks, plainpages]{hyperref}
\usepackage{mathtools}

\newtheorem{Thm}{Theorem}[section]
\newtheorem{Lem}[Thm]{Lemma}
\newtheorem{Pro}[Thm]{Proposition}
\newtheorem{Cor}[Thm]{Corollary}

\theoremstyle{definition}

\theoremstyle{remark}
\newtheorem{Rem}[Thm]{Remark}

\newcommand{\R}{\mathbb{R}}

\newcommand{\al}{\alpha}

\renewcommand{\phi}{\varphi}

\newcommand{\tr}{\operatorname{tr}}

\newcommand{\vol}{\operatorname{Vol}}

\renewcommand{\d}{\partial}

\begin{document}

\title[On harmonic and homogeneous almost \textbf{$\al$}-cosymplectic \textbf{$3$}-manifolds]{Some Solitons on Homogeneous Almost \textbf{$\al$}-Cosymplectic \textbf{$3$}-Manifolds and Harmonic Manifolds}

\author[N. A. Pundeer]{Naeem Ahmad Pundeer}
\address{Department of Mathematics\\
Jadavpur University\\
Kolkata-700032, India.}
\email{pundir.naeem@gmail.com}

\author[P. Ghosh]{Paritosh Ghosh}
\address{Department of Mathematics\\
Jadavpur University\\
Kolkata-700032, India.}
\email{paritoshghosh112@gmail.com}

\author[H. M. Shah]{Hemangi Madhusudan Shah}
\address{Harish-Chandra Research Institute\\
A CI of Homi Bhabha National Institute\\
Chhatnag Road, Jhunsi, Prayagraj-211019, India.}
\email{hemangimshah@hri.res.in}

\author[A. Bhattacharyya]{Arindam Bhattacharyya}
\address{Department of Mathematics\\
Jadavpur University\\
Kolkata-700032, India}
\email{bhattachar1968@yahoo.co.in}


\begin{abstract}
In this paper, we investigate the nature of Einstein solitons, whether it is steady, shrinking or expanding on almost $\alpha$-cosymplectic $3$-manifolds. We also prove that a simply connected homogeneous almost $\alpha$-cosymplectic $3$-manifold, admitting a contact Einstein soliton, is an unimodular semidirect product Lie group. Finally, we show that a harmonic manifold admits a  non-trivial Ricci soliton if and only if it is flat. Thus we show that rank one symmetric spaces of compact as well as non-compact type are stable under a Ricci soliton. In particular, we obtain a strengthening of Theorem $1$ and Theorem $2$ of \cite{B.15}.
\end{abstract}
\maketitle

\section{\textbf{Introduction}}
\hspace{0.7cm} The study of solitons, in particular Ricci solitons, on Riemannian manifolds play a vital role in understanding the geometry of underlying manifold. It is very interesting to study Ricci and Einstein solitons on almost $\alpha$-cosymplectic $3$-manifolds. Recently, Jin and Ximin \cite{Li} showed that a simply connected homogeneous almost $\alpha$-cosymplectic $3$-manifold, admitting  contact Ricci solitons, is cosymplectic; and the manifold under consideration is an unimodular semidirect product Lie group $\mathbb{R}^2\rtimes_A\mathbb{R}$, where $A=\left(\begin{array}{cc}
  0 & b \\
-b & 0
  \end{array}\right)$, equipped with a flat left invariant cosymplectic structure.\par
Motivated by this result we show in this paper that, if a simply connected homogeneous almost $\alpha$-cosymplectic $3$-manifold, with some 
additional hypothesis, admits a contact Einstein soliton, then the manifold is an unimodular semidirect product Lie group $G$ of type $G_{0b\overline{b}}=\mathbb{R}^2\rtimes_A\mathbb{R}$, where $A=\left(\begin{array}{cc}
      0 & b \\
     -b & 0
 \end{array}\right)\neq 0$.
  And also $G$ is the Lie group $\tilde{E}^2$ equipped with its flat left invariant cosymplectic structure (see Corrollary \ref{cor}). In order to prove this result, we first obtain a characterization of almost $\alpha$-cosymplectic $3$-manifold admitting contact Einstein solitons, which is the main theorem (Theorem \ref{th5}) of Section \ref{sec3}. To establish this aforementioned theorem we derive an identity (Lemma \ref{p1}) involving scalar curvature, Lie derivative of the metric and Ricci operator on a Riemannian manifold admitting Einstein soliton. We also give some conditions on $\alpha$ for contact Einstein solitons to be steady, shrinking or expanding on almost $\alpha$-cosymplectic $3$-manifolds (see Theorem \ref{th3}).\par

Another interesting topic in the differential geometry is the geometry of harmonic manifolds.
In this paper, we prove that 
a harmonic manifold admits a non-trivial
Ricci soliton if and only if it is flat. 
The flat harmonic manifold admits Ricci solitons of steady, expanding or shrinking type. We also determine the corresponding potential function. In fact, Busemann function on $\mathbb{R}^n$ turns to be the potential function in case of steady solitons (see Theorem \ref{main} of Section \ref{sec4}).\\
Note that any rank one symmetric space is harmonic.
Therefore, in particular, we obtain that there 
are no non-trivial Ricci solitons on  rank one symmetric spaces. Thus  harmonic manifolds, in particular, rank one symmetric spaces are stable 
under a Ricci  soliton.  It is  shown in Theorem $1$ and Theorem $2$ of \cite{B.15} that  any 
small perturbation of the non-compact symmetric metric is flown back to the original metric under 
an appropriately rescaled Ricci flow. Thus we obtain 
the strengthening of this result in case of 
non-compact rank one symmetric spaces.
Moreover, we also obtain that compact rank one 
symmetric spaces are  stable 
under the Ricci solitons.\\

\par
The paper is divided into four sections. Section \ref{sec2} is devoted to the preliminaries about Ricci soliton, Einstein soliton, almost $\alpha$-cosymplectic $3$-manifolds and harmonic manifolds. In Section \ref{sec3}, we prove our main results on almost $\alpha$-cosymplectic $3$-manifold admitting contact Einstein solitons, as stated above. In the last section, we prove the main result about harmonic manifolds admitting Ricci solitons.

\section{Preliminaries}\label{sec2}
In this section, we discuss some notions required to prove the results of this paper.

\subsection{Ricci solitons}
Ricci solitons are the self similar solutions of the Ricci flow. The concept of Ricci flow was first introduced by Hamilton \cite{Ha2} in (1982), motivated by the work of Eells and Sampson \cite{Eells} on harmonic map and the flow was given by the equation
 \begin{equation*}
   \frac{\d g}{\d t}=-2S,
 \end{equation*}
where $S$ is the Ricci tensor.\\
{\it Ricci solitons} are the generalizations of the Einstein metrics and are the solutions of the equation
\begin{equation}\label{eq21}
     Ric(g)+\frac{1}{2}\mathfrak{L}_{X}g=\lambda g,
 \end{equation}
 where $Ric(X, Y)=S(X, Y)$ is the Ricci curvature tensor, $\mathfrak{L}_X$ is the Lie derivative along the direction of the vector field $X$ and $\lambda$ is a real constant. The soliton is said to be {\it shrinking} if $\lambda>0$, {\it steady} if $\lambda=0$ and {\it expanding} if $\lambda<0$.\\
 \noindent Tashiro \cite{P Petersen} proved very important result
 for complete Einstein manifolds admitting Ricci solitons. 
 
\begin{Thm} \label{Tash}\cite{P Petersen}
  Let $(M, g)$ be a complete Riemannian $n$-manifold admitting a nontrivial function $f$ such that $\operatorname{Hess}f=\lambda g$, then $(M,g)$ is isometric to a complete warped product metric and must have one of the three forms:
\begin{enumerate}
  \item $M=\mathbb{R}\times N, ~~g=dr^2+\rho^2(r)g_N$,
  \item $M=\mathbb{R}^n, ~~g=dr^2+\rho^2(r)ds_{n-1}^2, ~~r\geq0$,
  \item $M= S^n, ~~g=dr^2+\rho^2(r)ds_{n-1}^2, ~~r\in[a,b]$.
\end{enumerate}
\end{Thm}
\subsection{Einstein solitons}
The  {\it Einstein solitons} are the generalization of the Ricci solitons, was first introduced by Catino and Mazzieri \cite{Cat} in (2016). They are the solutions of the equation
\begin{equation}\label{eq9}
  \mathfrak{L}_Vg+2S=(2\lambda+r)g,
\end{equation}
where, Ricci tensor $S(X,Y)=g(X,QY)$, $Q$ being the Ricci operator, $r$ is the scalar curvature, $\lambda\in\mathbb{R}$ is a constant and $V$ is known as \textit{potential vector field}.\\
Einstein solitons are the self-similar solutions of the Einstein flow,
\begin{equation*}
  \frac{\partial}{\partial t}g+2S=rg.
\end{equation*}
It is said to be \textit{steady} if $\lambda=0$, \textit{shrinking} if $\lambda>0$ and \textit{expanding} if $\lambda<0$.

\subsection{Almost contact metric manifolds}
In order to define contact metric manifolds, we need the concept of Reeb vector field.\newline
{\bf Reeb vector field \cite{Bla}:} A global vector field $\xi$ on a contact manifold $M^{2n+1}$, equipped with a global $1$-form $\eta$, is called {\it Reeb vector field} or {\it characteristic vector field}, if 
any vector field $X$ satisfies $\eta(\xi)=1$ and $d\eta(X, \xi)=0$.\\\\
{\bf Almost contact manifold \cite{Bla}:}
Let $M$ be a Riemannian manifold of dimension $(2n+1)$, $n\geq1$. $M^{2n+1}$ is said to have an \textit{almost contact structure} $(\phi, \xi, \eta)$, if there exists a $(1,1)$-tensor $\phi$, a global vector field $\xi$ and a $1$-form $\eta$ such that
\begin{eqnarray}\label{eq1}
  \phi^{2}X = -X+\eta(X)\xi,~~~~
  \eta(\xi) = 1,
\end{eqnarray}
for any vector field $X$ on $M$, where $\xi$ is the \textit{Reeb vector field}. The manifold $M$ equipped with the structure $(\phi, \xi, \eta)$ is called an {\it almost contact manifold}.\\\\
{\bf Almost contact metric manifold \cite{Bla}:} 
A Riemannian metric $g$ is said to be \textit{compatible} with an almost contact structure $(\phi, \xi, \eta)$, if
\begin{equation}\label{eq2}
  g(\phi X, \phi Y)=g(X, Y)-\eta(X)\eta(Y),
\end{equation}
holds for any $X, Y \in\chi(M)$ and $(M, \phi, \xi, \eta, g)$ is called an \textit{almost contact metric manifold}.\\\\
{\bf Normal almost contact metric manifold \cite{Bla}:}
An almost contact metric manifold is said to be \textit{normal}, if for any $X, Y\in\chi(M)$ the tensor field $N=[\phi,\phi]+2d\eta\otimes\xi$ vanishes everywhere on the manifold, where $[\phi,\phi]$ is the Nijenhuis tensor of $\phi$.\\\\
{\bf Homogeneous almost contact metric manifold \cite{Li}:}
An almost contact metric manifold $(M, \phi, \xi, \eta, g)$ is said to be {\it homogeneous}, if there exists a connected Lie group $G$ of isometries acting transitively on $M$ leaving $\eta$ invariant.\\

\subsection{Cosymplectic manifolds}
A $(2n + 1)$-dimensional manifold is said to be a {\it cosymplectic manifold} \cite{Lib}, if it admits a closed, 1-form $\eta$ and $2$-form $\Phi$ such that $\eta\wedge\Phi^n$ is a volume element, where $\Phi(X,Y)=g(\phi X,Y)$ is a $2$-form on $M^{2n+1}$.\\\\
{\bf Almost cosymplectic manifold \cite{Lib}:}
If $\eta$ and $\Phi$ are not closed but $\eta\wedge\Phi^n$ is a volume form, then the manifold is called \textit{almost cosymplectic manifold}.\\\\
{\bf $\alpha$-cosymplectic manifold \cite{Pe}:}
An almost cosymplectic manifold is said to be $\alpha$-cosymplectic if $d\eta=0$ and $d\Phi=2\alpha\eta\wedge\Phi$ for some constant $\alpha$.\\\\
{\bf Almost $\alpha$-cosymplectic manifold \cite{Lib}:}
An \textit{almost $\alpha$-cosymplectic manifold} is defined as an almost contact metric manifold with $d\eta=0$ and $d\Phi=2\alpha\eta\wedge\Phi$, for any constant $\alpha$. In particular, the almost $\alpha$-cosymplectic manifold is
\begin{itemize}
  \item \textit{almost $\alpha$-Kenmotsu} if $\alpha\neq0$,
  \item \textit{almost cosymplectic} if $\alpha=0$,
  \item \textit{almost Kenmotsu} if $\alpha=1$.
\end{itemize}
{\bf Harmonic vector field \cite{Perr}:}
A characteristic vector field $\xi$ on an almost $\alpha$-cosymplectic manifold is {\it harmonic} if and only if $\xi$ is an eigenvector field of the Ricci operator $Q$.

\subsection{Almost $\alpha$-cosymplectic $3$-manifold}
In this article, we will mainly focus  on $3$-dimensional 
almost $\alpha$-cosymplectic manifold.  In what follows, we will be using the following results.   

\begin{Thm}\cite{Pe}
  An almost $\alpha$-cosymplectic $3$-manifold is $\alpha$-cosymplectic if and only if $\mathfrak{L}_\xi h=0$, where $h=\frac{1}{2}\mathfrak{L}_\xi\phi$.
  \end{Thm}
  
\noindent
Any almost $\alpha$-cosymplectic $3$-manifold satisfies important 
relationships between  $\Phi, \xi$ and $h$.

\begin{Lem}\cite{Pe}
  Let $M^{2n+1}$ be an almost $\alpha$-cosymplectic $3$-manifold, then  we have,
  \begin{eqnarray}\label{eq3}
  \nabla_\xi\phi = 0,~~~~
  \nabla\xi = 0, ~~~~
  h\phi+\phi h = 0,~~~~
  h\xi = 0,
\end{eqnarray}
with
\begin{equation}\label{eq4}
  \nabla_X\xi=-\alpha\phi^2X-\phi hX.
\end{equation}
\end{Lem}
\noindent
We would require some identities on the $\phi$-bases \cite{Bla} and the following table of the Levi-Civita connection. 
\begin{Pro}\cite{Pe}
  On almost $\alpha$-cosymplectic $3$-manifold, there exists $\phi$-bases satisfying 
  \begin{eqnarray*}
  he=\sigma e,~~~~
  h\phi e=-\sigma\phi e,~~~~
  h\xi=0,
  \end{eqnarray*}
  with $\sigma$ a local smooth eigen-function of $h$.
\end{Pro}
\begin{Thm}\cite{Pe}
  The Levi-Civita connection on almost $\alpha$-cosymplectic $3$-manifold are given by,
  \begin{equation}\label{eq5}
    \begin{cases}
       & \nabla_ee=-a\phi e-\alpha\xi,~~~~ \nabla_{\phi e} e=-b\phi e+\sigma\xi,~~~~ \nabla_\xi e=\mu\phi e,  \\
       & \nabla_e\phi e=ae+\sigma\xi,~~~~ \nabla_{\phi e} \phi e=be-\alpha\xi,~~~~ \nabla_\xi \phi e=-\mu e,  \\
       & \nabla_e\xi =\alpha e-\sigma\phi e,~~~~ \nabla_{\phi e}\xi=-\sigma e+\alpha\phi e,~~~~ \nabla_\xi\xi=0,
    \end{cases}
  \end{equation}
   where $a=g(\nabla_e\phi e, e)$, $b=-g(\nabla_{\phi e}e, \phi e)$ and $\mu=g(\nabla_\xi e, \phi e)$ are smooth functions.
\end{Thm}

\noindent 
The Ricci operator on almost $\alpha$-cosymplectic $3$-manifold is known
explicitly \cite{Pe}.

\begin{Pro}\cite{Pe}
  The Ricci operator $Q$ on almost $\alpha$-cosymplectic $3$-manifold is given by,
  \begin{equation}\label{eq7}
	\begin{cases}
		&Q\xi=-(2\alpha^2+\operatorname{tr}h^2)\xi+(2b\sigma-e(\sigma))\phi e-(2a\sigma+(\phi e)(\sigma))e,\\
		&Q\phi e=(2b\sigma-e(\sigma))\xi+(\alpha^2+\frac{r}{2}+\frac{\operatorname{tr}h^2}{2}+2\sigma\mu)\phi e+(\xi(\sigma)+2\alpha\sigma)e,\\
		&Q e=-(2a\sigma+(\phi e)(\sigma))\xi+(\xi(\sigma)+2\alpha\sigma)\phi e+(\alpha^2+\frac{r}{2}+\frac{\operatorname{tr}h^2}{2}-2\sigma\mu)e.\\	
	\end{cases}
\end{equation}
Furthermore, the scalar curvature $r=\operatorname{tr}Q$ is given by
\begin{equation}\label{eq8}
  r=-6\alpha^2-\operatorname{tr}h^2-2(a^2+b^2)-2(\phi e)(a)+2e(b).
\end{equation}
\end{Pro}

\vspace{0.1in}

\noindent 
The structure of simply-connected, homogeneous 
almost $\alpha$-cosymplectic $3$-manifold, admitting a contact Ricci soliton, is very well known. 
\begin{Thm}\label{th2}\cite{Li}
  Let $M$ be a simply-connected, homogeneous almost $\alpha$-cosymplectic $3$-manifold admitting a contact Ricci soliton. Then $M$ is an unimodular semidirect product Lie group $G$ of type $G_{0b\overline{b}}=\mathbb{R}^2\rtimes_A\mathbb{R}$, where
  $A=\left(\begin{array}{cc}
  0 & b \\
  -b & 0
  \end{array}\right)$, equipped with a flat left invariant cosymplectic structure.
  Moreover, we have the following:
  \begin{enumerate}
    \item If $A=0$, i.e., $b=0$, $G$ is the abelian Lie group $\mathbb{R}^3$ equipped with its flat left invariant cosymplectic structure.
    \item If $A\neq0$, i.e., $b\neq 0$, $G$ is the Lie group $\tilde{E}^2$ equipped with its flat left invariant cosymplectic structure.
  \end{enumerate}
\end{Thm}

\vspace{0.1in}

\subsection{Harmonic manifolds}
\indent
A complete Riemannian manifold $(M^{n},g)$ is said to be {\it harmonic}, if for any $p \in M$, the volume density $\omega_{p}(q) = \sqrt{\det (g_{ij} (q))}$ in normal coordinates, centered at  any $p \in M$ is a radial function \cite{AL Besse}.
Thus, $$\Theta(r) = r^{n-1} \sqrt{\det (g_{ij} (q))},$$ is density of geodesic
sphere, is a radial function.
It is known that harmonic manifolds are Einstein \cite{AL Besse}. They are naturally classified as per the sign of the Ricci constant. Let $r$ be the constant scalar curvature of $M$.
\begin{itemize}
\item If $r = 0$, then $M$ is flat, that is 
$(M,g)=(\mathbb{R}^{n}, Can)$ (Lemma \ref{flat}).
\item If $r>0$, then by Bonnet-Myer's theorem $M$ is compact with finite fundamental group. They are compact rank one symmetric spaces by a well known result of Szabo (cf. \cite{RS.03}).
\item If $r<0$, then $M$ is non-compact harmonic manifold. They are rank one symmetric spaces of non-compact type, if dimension of $M$ is atmost $5$.
\end{itemize}

\vspace{0.1in}

The main result in the theory of harmonic spaces is the Lichnerowicz Conjecture: {\it Any simply connected, complete harmonic manifold is either flat or a rank one symmetric space.} By the above classification, we see that the conjecture is resolved for compact harmonic manifolds and is open for non-compact harmonic manifolds of dimension $6$. There are counter examples to the conjecture when dimension is atleast $7$, known as the Damek-Ricci spaces or NA spaces. See for more details references in \cite{RS.03}.\\

In the category of non-compact harmonic manifolds, we will be considering
 simply connected, complete, non-compact  harmonic manifolds. It follows that,  these spaces don't have conjugate points (cf. \cite{RS.03}). Hence, by the Cartan-Hadamard  theorem,
 $$\exp_{p} : T_{p}M \rightarrow M$$ is a diffeomorphism
 and every geodesic of $M$ is a line. That is, if $\gamma_{v} :  \R  \rightarrow M$
 is a geodesic of $M$ with $v \in S_{p}M$, $\gamma_{v}'(0) = v$,
 then  $d(\gamma_{v}(t),  \gamma_{v}(s)) = |t-s|.$\\

\noindent
 {\bf Busemann function:}
 Let $\gamma_{v}$ be a geodesic line, then the two {\it Busemann functions}
 associated to $\gamma_{v}$ are defined as  \cite{P Petersen}:
 $$b_{v}^{+}(x) = \lim_{t \rightarrow \infty} d(x, \gamma_{v}(t)) - t,$$
 $$ b_{v}^{-}(x) = \lim_{t \rightarrow -\infty} d(x, \gamma_{v}(t)) - t.$$

\section{Einstein Solitons on Almost $\alpha$-Cosymplectic $3$-Manifolds}\label{sec3}
In this section, we examine the nature of a {\it contact Einstein soliton} 
on almost $\alpha$-cosymplectic manifold. We also show that, the
characteristic vector field $\xi$ is harmonic on
almost $\alpha$-cosymplectic $3$-manifold admitting a contact Einstein soliton. Finally, we generalize  Theorem \ref{th2} using these results.\\

\noindent
{\bf Contact Einstein soliton:}
Let $(M^{2n+1},g)$ be a Riemannian manifold of dimension $2n+1$ $(n\geq 1)$. Consider the Einstein soliton (\ref{eq9}), with potential vector field $V$, on an almost contact metric manifold $(M,\phi,\xi,\eta,g)$. Then the soliton is called \textit{contact Einstein soliton}, if $V=\xi$  that is, the potential vector field is the characteristic vector field.\\\\
  The potential vector field $V$ is called \textit{transversal}, if it is orthogonal to the characteristic vector field, that is $V\perp\xi$.

\begin{Thm}\label{th3}
  Let $(M, \phi, \xi, \eta, g)$ be an almost $\alpha$-cosymplectic $3$-manifold, admitting a contact Einstein soliton. Then the soliton 
  is:
  \begin{enumerate}
    \item \textit{steady}, if 
    $\alpha^2=\sigma^2-(a^2+b^2)-(\phi e)(a)+e(b)$,
    \item \textit{shrinking}, if $\alpha^2>\sigma^2-(a^2+b^2)-(\phi e)(a)+e(b)$,
    \item \textit{expanding}, if $\alpha^2<\sigma^2-(a^2+b^2)-(\phi e)(a)+e(b)$.
  \end{enumerate}
\end{Thm}
\begin{proof}
  If the soliton is contact Einstein soliton, using $V=\xi$ in (\ref{eq9}), we have
  \begin{equation}\label{eq10}
    g(\nabla_X\xi, Y)+g(X, \nabla_Y\xi)+2g(X, QY)=(2\lambda+r)g(X, Y),
  \end{equation}
  for any vector fields $X, Y$ on $M$.\\
 Substituting $X=Y= \xi$ in the above equation and using (\ref{eq7}), we obtain
  \begin{equation}\label{eq11}
    \lambda=-2\alpha^2-2\sigma^2-\frac{r}{2}.
  \end{equation}
  From the expression of $r$ (\ref{eq8}), we get
  \begin{equation}\label{eq12}
    \lambda=\alpha^2-\sigma^2+(a^2+b^2)+(\phi e)(a)-e(b),
  \end{equation}
  from which we can conclude the proof.
\end{proof}

\begin{Thm}\label{th4}
  Let $(M, \phi, \xi, \eta, g)$ be an almost $\alpha$-cosymplectic $3$-manifold, admitting a contact Einstein soliton. Then the characteristic vector field $\xi$ is harmonic.
\end{Thm}
\begin{proof}
  From (\ref{eq10}), we get for $X=\xi$ and $Y=e$,
  \begin{equation}\label{eq13}
    (\phi e)(\sigma)=-2a\sigma.
  \end{equation}
  And for $X=\xi$ and $Y=\phi e$, from (\ref{eq10}) we have 
  \begin{equation}\label{eq14}
    e(\sigma)=2b\sigma.
  \end{equation}
  Now, using (\ref{eq13}) and (\ref{eq14}) in the expression of $Q\xi$ in (\ref{eq7}), we obtain
  \begin{equation*}
    Q\xi=-(2\alpha^2+2\sigma^2)\xi,
  \end{equation*}
  which shows that $\xi$ is an eigenvector field of the Ricci operator $Q$ concluding the fact that $\xi$ is harmonic.
\end{proof}
\noindent
We derive the identity involving the Lie derivative of the metric, Ricci operator, the potential vector field $V$.  

\begin{Lem}\label{p1}
  Let $(M, g)$ be a Riemannian manifold of scalar curvature $r$, admitting an Einstein soliton (\ref{eq9}). Then
    \begin{equation}\label{eq20}
      \Vert\mathfrak{L}_Vg\Vert^2=2dr(V)+4\operatorname{div}\bigg(\bigg(\lambda+\frac{r}{2}\bigg)V-QV\bigg),
    \end{equation}
    where $Q$ is the Ricci operator.
\end{Lem}
\begin{proof}
  In local coordinate system,  (\ref{eq9}) leads to
  \begin{equation*}
    \mathfrak{L}_Vg^{ij}+S^{ij}=(2\lambda+r)g^{ij}.
  \end{equation*}
  Therefore,
  \begin{align}\label{e21}
    \Vert\mathfrak{L}_Vg\Vert^2 =&-S^{ij}\mathfrak{L}_Vg_{ij}+(2\lambda+r)g^{ij}\mathfrak{L}_Vg_{ij}. \nonumber \\
    =& -\mathfrak{L}_Vr + g_{ij}\mathfrak{L}_VS^{ij}-(2\lambda+r)g_{ij}\mathfrak{L}_Vg^{ij}.
  \end{align}
  Now,
  \begin{align}\label{e22}
    g_{ij}\mathfrak{L}_VS^{ij}=&g_{ij}\nabla_VS^{ij}-g_{ij}\nabla_\alpha V_iS^{\alpha j}-g_{ij}\nabla_\alpha V_jS^{i\alpha} \nonumber\\
    =& 2dr(V)-2\operatorname{div}QV.
  \end{align}
 Observing that $g_{ij}\mathfrak{L}_Vg^{ij}=-2\operatorname{div}V$ and using  (\ref{e21}) and (\ref{e22}), we get the required result.
\end{proof}

\noindent
Now we derive the main result of this section.
\begin{Thm}\label{th5}
  Consider $M$ to be an almost $\alpha$-cosymplectic $3$-manifold, admitting a contact Einstein soliton. Then the following hold.
  \begin{enumerate}
    \item If $\sigma \neq 0$, then $\alpha=a^2+b^2-2\lambda^2+(\phi e)(a)-e(b)$.
    \item If $\sigma=0$, then $M$ is cosymplectic.
  \end{enumerate}
\end{Thm}
\begin{proof}
Replacing $X$ by $e$ and $Y$ by $\phi e$, from (\ref{eq10}) we get
  \begin{equation*}
     g(\nabla_e\xi, \phi e)+g(e, \nabla_{\phi e}\xi)+2g(e, Q\phi e)=(2\lambda+r)g(e, \phi e).
  \end{equation*}
  Using (\ref{eq5}) and (\ref{eq7}), after simplification we acquire,
  \begin{equation}\label{eq15}
    \xi(\sigma)=\sigma-2\alpha\sigma.
  \end{equation}
  Now putting $X=e=Y$ in (\ref{eq10}) and using (\ref{eq5}), (\ref{eq7}), (\ref{eq8}) and (\ref{eq12}), we get
  \begin{equation}\label{eq16}
    6\alpha^2+6\sigma^2-4\sigma\mu+2\alpha+r=0.
  \end{equation}
  Similarly, putting $X=\phi e=Y$ in (\ref{eq10}) and using (\ref{eq5}), (\ref{eq7}), (\ref{eq8}) and (\ref{eq12}), we also obtain
  \begin{equation}\label{eq17}
    6\alpha^2+6\sigma^2+4\sigma\mu+2\alpha+r=0.
  \end{equation}
  So  comparing (\ref{eq16}) and (\ref{eq17}), we have $\sigma\mu=0$.
  \noindent
  If $\sigma\neq 0$, then from (\ref{eq17}), we obtain the required result using (\ref{eq8}).\\
  Now suppose $\sigma=0$, then $M$ is $\alpha$-cosymplectic.
  From \cite{Pe}, recall that an almost $\alpha$-cosymplectic manifold $M$ is $\alpha$-cosymplectic if and only if for any $X\in\chi(M)$,
  \begin{equation}\label{eq18}
    QX=\bigg(\alpha^2+\frac{r}{2}\bigg)X-\bigg(3\alpha^2+\frac{r}{2}\bigg)\eta(X)\xi.
  \end{equation}
  Since $\nabla\xi$ is symmetric, (\ref{eq10}) becomes
  \begin{equation}\label{eq19}
    g(\nabla_X\xi, Y)+g(X, QY)=\bigg(\lambda+\frac{r}{2}\bigg)g(X, Y).
  \end{equation}
  Using (\ref{eq4}) and (\ref{eq18}), we have from (\ref{eq19}), for any $X, Y\in\chi(M)$,
  \begin{equation*}
    (\alpha^2+\alpha-\lambda)g(X, Y)-\bigg(3\alpha^2+\alpha+\frac{r}{2}\bigg)\eta(X)\eta(Y)=0,
  \end{equation*}
  which implies $\alpha^2+\alpha-\lambda=0$ and $3\alpha^2+\alpha+\frac{r}{2}=0$.\\
  That is $\lambda=\alpha^2+\alpha$ and $r=-6\alpha^2-2\alpha=\mbox{constant}$, so that, $\lambda+\frac{r}{2}=-2\alpha^2$.\\
  Also, from (\ref{eq18}), we have $Q\xi=-2\alpha^2\xi$ which implies $(\lambda+\frac{r}{2})\xi-Q\xi=0$.
 Therefore, using Lemma \ref{p1}  (\ref{eq20}), we can say that $\xi$ is a Killing vector field, that is, $\nabla\xi$ is skew-symmetric. But in our case $\nabla\xi$ is symmetric, which implies $\nabla\xi=0$, that is, $\alpha=0$, proving the fact that $M$ is cosymplectic.

\end{proof}

\begin{Cor}\label{cor}
  Consider $M$ to be a simply-connected, homogeneous, almost $\alpha$-cosymplectic $3$-manifold, admitting a contact Einstein soliton with $\sigma=0$. Then $M$ is an unimodular semidirect product Lie group $G$ of type $G_{0\mu\overline{\mu}}=\mathbb{R}^2\rtimes_A\mathbb{R}$, where $A=\left(\begin{array}{cc}
                                                                                                                 0 & \mu \\
                                                                                                                 -\mu & 0
                                                                                                               \end{array}\right)\neq 0$, is a real matrix.
  Moreover, $G$ is the Lie group $\tilde{E}^2$ equipped with its flat left invariant cosymplectic structure.
\end{Cor}
\begin{proof}
  The proof follows from Theorem \ref{th2} and Theorem \ref{th5}.
\end{proof}

\section{Ricci Solitons on Harmonic Manifolds}\label{sec4}
Recall that the Ricci solitons are solutions of (\ref{eq21}).
Clearly, if a manifold is Einstein of constant $r$, then 
trivial solitons $X=0$ and $X$ a  Killing  vector field are solutions of \eqref{eq21} with $\lambda=r$. \\\\
In this section, we study Ricci solitons on complete, simply connected, harmonic manifolds. We prove a {\it Lichnerowicz type result that, a harmonic manifold admits a non-trivial Ricci soliton if and only if $M$ is flat}. More precisely, we show that compact harmonic manifolds and non-flat harmonic manifolds {\it do not} admit non-trivial Ricci solitons. But flat  harmonic manifolds {\it do admit} non-trivial shrinking
and expanding Ricci solitons.\\

\noindent
In the sequel, harmonic manifold means complete, simply connected harmonic 
manifold. \\

\noindent
The main theorem of this section is:
\begin{Thm}\label{main}
Let $(M,g)$ be a harmonic manifold. Then $M$ admits a non-trivial Ricci soliton if and only
if $M$ is flat. In this case, the steady Ricci soliton is trivial of Killing type given by $X=\nabla {b_{v}}^{-};$ where $b_{v}^{-}(x) = - \langle x, v \rangle$, the Busemann function, is the potential function on $M$. In case, the Ricci soliton is shrinking or expanding,
the potential function is given by $f(x) = \lambda {d(x,p)}^2 + f(p)$, for
constant $\lambda \neq 0$; and point
$p$ is the minimum or the maximum of $f$ and $X = \nabla f$ is the
corresponding non-trivial Ricci soliton.
\end{Thm}

\begin{Cor}
There are no deformations of harmonic manifolds, and in particular, of rank one symmetric spaces under a Ricci soliton.
In particular, we obtain a strengthening of 
\cite{B.15}, and also the stability of the compact
rank one symmetric spaces  under a Ricci soliton.
\end{Cor}

\vspace{0.1in}

\subsection{Proof of Theorem \ref{main}}
In this subsection we prove Theorem \ref{main}. 
We begin with the following important proposition.

\begin{Pro}
If a complete manifold admits a  Ricci soliton, then it is a gradient soliton.
\end{Pro}
\begin{proof}
  This follows from  Remark 3.2 of Perelman \cite{Perelman}
  (see also page $2$ of \cite{A.A.})). Hence, in this case,
  if $X$ is a Ricci soliton, then $X=\nabla f$, for some smooth function $f : M \rightarrow R$.
  \end{proof}

\noindent
\begin{Rem}
 Here we are only concerned with simply connected and complete Riemannian
manifold. In this case, clearly,  we can write  $X=\nabla f$ by Poincar\'{e} Lemma, for some $f\in C^\infty(M)$.
\end{Rem}

\begin{Lem}\label{flat}
 Ricci flat harmonic manifold is flat.
 \end{Lem}
 \begin{proof}
 It can be shown that any harmonic manifold $(M,g)$ is asymptotically harmonic
 \cite{RS.03}.  That is there exists a constant $h \geq 0$ such that
 $$\Delta {b_{v}}^{\pm} = h.$$
 Let $L_{t} = {\nabla}^2 {b_{v}}^{+}$ denote the second fundamental form of horosphere,
 $b_{v}^{-1}(t)$.  Then $L_{t}$ satisfies the Riccati equation, that is for
 $x_{t} \in \{\gamma'(t)\}^{\perp}$,
 $$ {L_{t}}' (x_{t}) + {L_{t}}^2(x_{t}) +
 R(x_{t},\gamma'(t))\gamma'(t)= 0.$$
 Tracing the above equation, we obtain that $\tr {L_{t}}^2 = 0$, as Ricci$(\gamma'(t),\gamma'(t)) = 0$.
 But as $L_{t}$ is a symmetric operator on 
 $\{\gamma'(t) \}^{\perp}$,
 $L_{t} = 0$.  Consequently,  $R(x, v) v = 0$ for any $x \in {v}^{\perp}$ and for any $v \in SM$.
 Thus $(M, g)$ is flat.
 \end{proof}

 \begin{Pro}\label{prop8}
If a harmonic manifold admits a Ricci soliton, then it admits a Gaussian. 
 \end{Pro}
 \begin{proof}
As in this case $(M,g)$ is Einstein,  it follows that
 \begin{equation}\label{eq23}
     \nabla^{2}f=2(\lambda-r)g,
 \end{equation}
 \noindent where $r$ is a constant scalar curvature of $M$.
 Thus $f$ is a Gaussian, that is it satisifes (\ref{eq23}).
 \end{proof}
 
 \vspace{0.1cm}

 \begin{Lem}\label{lem10}
Let $X=\nabla f$ be a Killing vector field on compact harmonic manifold, then $X$ is trivial. Trivial solitons of Killing type do not exist on non-compact, non-flat harmonic manifold. On flat harmonic manifold,  Killing vector field is $X=\nabla {b_{v}}^{-}$, where
 $b_{v}^{-}(x) = - \langle x, v \rangle$ is a Busemann function on $\mathbb{R}^{n}$.
 \end{Lem}
 \begin{proof}
 Because $X=\nabla f$ is a non-trivial Killing vector field, we have
 \begin{equation*}
     \nabla^{2}f=0.
 \end{equation*}
 \noindent Therefore, $\rVert \nabla f \rVert = \mbox{constant} \neq 0$,
 consequently, $f$  has no critical points.
\noindent Any Killing vector field of constant norm satisfies 
(p. 164-167, \cite{P Petersen}):
\begin{eqnarray*}
	{\rVert \nabla^{2}f \rVert}^{2} =\mbox{Ric}(\nabla f, \nabla f).
	\end{eqnarray*}
Therefore,
\begin{eqnarray*}
0 = \rVert \nabla^{2} f \rVert = r{\rVert \nabla f \rVert}^{2}
\end{eqnarray*}
 \noindent    This implies that for $f$ non-constant, $r=0$
 and therefore $\mbox{Ric} \equiv 0$ and hence harmonic manifold must be flat (Lemma \ref{flat}).\\
 \noindent We have $\rVert \nabla f \rVert = \mbox{constant}$.
We may assume that  $\rVert \nabla f \rVert = 1$,
therefore $f$ is distance function which is harmonic function on $({\mathbb{R}}^{n}, Can)$.
 By Proposition 5.1 of \cite{RS.03}, it follows that 
 $$f(x) = b_{v}^{-}(x) = - \langle x, v \rangle,$$ is a Busemann  function on ${\mathbb{R}}^n$ \cite{P Petersen}.\\
If $M$ is compact,  $\nabla^{2} f = 0$ implies that $f$ is a harmonic function.
Hence, $f$ must be a constant function.
 \end{proof}

\vspace{0.1in}

\begin{Pro}\label{com}
A  compact harmonic manifold $(M,g)$ does not admit a non-trivial  Ricci soliton.
\end{Pro}	
\begin{proof}
We have, $${\nabla}^2 f = 2 (\lambda - r) g.$$
Therefore,  
$\Delta f =  2 (\lambda - r) n$ implies by the Bochner's formula that,
\begin{eqnarray}\label{gradf1}
  \frac{1}{2}  \Delta(\| \nabla f \|^2) = 4 (\lambda - r)^2 n^2 + r (\| \nabla f \|^2).
 \end{eqnarray}
Therefore,
$$ 4 (\lambda - r)^2 n^2 \vol(M)   = - r \int_{M} \| \nabla f \|^2 < 0.$$
This implies that $\| \nabla f \| = 0$, therefore $f$ is constant.
\end{proof}	

\vspace{0.1in}

\begin{Lem}\label{main:lem}
A non-compact, harmonic manifold  admits a non-trivial Ricci  soliton if and only if it is flat. 
The flat harmonic manifold admits  shrinking and  expanding Ricci solitons with the corresponding potential function,
 $f(x) = \lambda  {d(p, x)}^2 + f(p)$, for some $p \in M$.
\end{Lem}
\begin{proof}
Suppose that  a non-compact, harmonic manifold  admits a non-trivial Ricci  soliton.
Therefore, it admits a Gaussian with 
$(\lambda - r) \neq 0$.
$${\nabla}^2 f = 2 (\lambda - r) g.$$
Therefore, $f$ is either convex or concave function. Consequently, the only
possible critical point of $f$ is either maximum or minimum of $f$.
Suppose that $p$ is a critical point of  $f$.
Note that along any unit speed geodesic of $M$ starting from $p$,
\begin{eqnarray}\label{rad}
f''(t) = 2 (\lambda - r).
\end{eqnarray}
Therefore,
$f'(t) =  2 (\lambda - r) t + c$. Hence, there is exactly one critical
point, and hence $c = 0$.
Thus, $f(t) = (\lambda - r) t^2 + f(p),$  consequently $f$ is a radial function.
This implies that,
$$\Delta f = f'' + \frac{{\Theta}'}{\Theta} f' = 2 (\lambda - r) n.$$
Therefore,
$$ f'' +  \frac{{\Theta}'}{\Theta} 2 (\lambda - r) t = 2 (\lambda - r) n.$$
Consequently by (\ref{rad}),
$$ \frac{{\Theta}'(t)}{\Theta(t)} = \frac{n-1}{t}.$$
Comparing with the series expansion (see (4.4) of \cite{RS.03}),
$$\frac{{\Theta}'(t)}{\Theta(t)} = \frac{n-1}{t} - \frac{r\;t}{3} + \cdots,$$
we obtain $r = 0$, hence $M$ is flat.
Finally,  $f(x) = \lambda  {d(p, x)}^2 + f(p)$ follows from
section $1$ of  \cite{CC.96}.
\end{proof}

\vspace{0.1in}

\noindent
Finally we come to the proof of  Theorem \ref{main}.
\\\\
{\bf Proof:}
 A  compact harmonic manifold can't admit non-trivial  Ricci soliton  (Proposition \ref{com}). 
If a non-compact harmonic manifold admits a trival 
Ricci soliton of Killing type, then $(\lambda - r) = 0$, implies that $r = 0.$ Therefore, $M$ is flat and $X=\nabla {b_{v}}^{-}$ (Lemma \ref{lem10}). 
If a non-compact harmonic manifold admits a non-trival 
Ricci soliton, then $(\lambda - r) \neq 0$ again implies that
$r = 0,$  and $M$ is flat. In this case $X = \nabla f$, where
$f(x) =  \lambda  {d(p, x)}^2 + f(p)$, for some $p \in M$
(Lemma \ref{main:lem}).

\vspace{0.1in}

\noindent
\begin{Rem}
We have shown that Theorem \ref{main} confirms Theorem \ref{Tash} in case of harmonic manifolds. Also Theorem \ref{main}  implies that there are no non-trivial deformation of non-flat harmonic  manifolds.  This indicates 
a result supporting the conjecture that, there are no non-trivial deformations of harmonic manifolds; and hence there should be only finitely many classes of harmonic manifolds.
\end{Rem}

\section{Acknowledgements}
 Dr. Naeem Ahmad Pundeer would like to thank to  U.G.C. for its Dr. D.S. Kothari Postdoctoral Fellowship.  The corresponding author, Mr. Paritosh Ghosh, thanks UGC Junior Research Fellowship of India. The authors also
 would like to thank Mr. Dipen Ganguly for his wishful help in this research.

\end{document}